\documentclass[11pt]{amsart}

\usepackage{amsmath,amssymb,amsthm}
\usepackage{amsmath,amssymb,amsthm,amscd}
\usepackage{enumerate}
\usepackage[frame,cmtip,arrow,matrix,line,graph,curve]{xy}
\usepackage{graphpap, color}
\usepackage[mathscr]{eucal}
\usepackage{color}
\numberwithin{equation}{section}
\newtheorem{prop}{Proposition}
\newtheorem{theo}[prop]{Theorem}
\newtheorem{lemm}[prop]{Lemma}

\newtheorem{rema}[prop]{Remark}
\newtheorem{defi}[prop]{Definition}

\theoremstyle{definition}

\newtheorem*{ack}{Acknowledgment}
\theoremstyle{remark}

\newcommand{\p}{\partial}

\begin{document}
\title{The rigidity of hypersurfaces in Euclidean space}

\author{Chunhe Li}
\address{School of Mathematical Sciences  \\ University of Electronic Science and Technology of China \\ Chengdu, China} \email{chli@fudan.edu.cn}
\author{Yanyan Xu}
\address{School of Mathematical Sciences\\ University of Electronic Science and Technology of China \\ Chengdu, China}
\email{XY1101UESTC@163.com}
\thanks{Research of the first author is partially supported  by NSFC Grant No.11571063.}
\begin{abstract}
In the present paper, we revisit the  rigidity of hypersurfaces in Euclidean space. We highlight Darboux equation and give new proof of rigidity of hypersurfaces by energy method and maximal principle.
\end{abstract}

\subjclass{53C24,53C45.}

\maketitle

\section{introduction}

The isometric embedding problem is one of the fundamental problems in differential geometry. Since Riemannian manifold was formulated by Riemmann in 1868,
naturally there arose the question of whether an abstract Riemmannian manifold
is simply a submanifold of some Euclidean space with its induced metric. In other
words, it's the question of reality of Riemannian manifold. (see more details in an expository note \cite{HH}.)

Mathematically, the isometric embedding problem is to solve the following system. For any given Riemannian manifold $(\mathcal{M},g)$, there is a surface $\vec{r}:\mathcal{M}\mapsto \mathbb{R}^{n+1}$ such that
\begin{eqnarray}\label{1.1}
d\vec{r}\cdot d\vec{r}=g,
\end{eqnarray}
where $\cdot$ denotes Euclidean inner product.
In the present paper we assume the $\vec{r}$ is a hypersurface, i.e. $\mathcal{M}$ is a manifold of $n$ dimension.

As known the uniqueness of solution in PDEs is related to the existence, hence it's another important topic. The counterpart of uniqueness in isometric embedding is global rigidity. The rigidity is to characterize isometric deformation of surfaces which is closely
related the global isometric embedding of surfaces.
\begin{defi}
An immersed surface $\vec{r}:\mathcal{M}\rightarrow \mathbb{R}^{3}$ is rigid if every immersion $\tilde{r}:\mathcal{M}\rightarrow \mathbb{R}^{3}$, with the same induced metric, is congruent to $\vec{r}$, that is, differs from $\vec{r}$ by an isometry of $\mathbb{R}^3.$
\end{defi}

If $\vec{r},\tilde{r}$ differs from  by an isometry of $\mathbb{R}^3,$ they are isometric naturally. Global rigidity says there is no other $\tilde{r}$ which is isometric to $\vec{r}$ except such trivial $\tilde{r}$ congruent to $\vec{r}$, hence global rigidity can be viewed as the uniqueness of the solution to isometric embedding problem.

The linearized version of global rigidity is infinitesimal rigidity. We say that $\vec{r}_t$ yields a first order isometric deformation of $\vec{r}=\vec{r}_0$ if the induced metric $g_t=d\vec{r}_t\cdot d\vec{r}_t$ has a critical point at $t=0,$
\begin{eqnarray}
\frac{d}{dt}(d\vec{r}_t\cdot d\vec{r}_t)=0,\text{ at } t=0.\nonumber
\end{eqnarray}
Set $\vec{\tau}=\frac{d r_t}{dt}$ at $t=0$, then the infinitesimal problem becomes
\begin{eqnarray}\label{1.2}
d\vec{r}\cdot d\vec{\tau}=0.
\end{eqnarray}

As known, the isometry group of $\mathbb{R}^{n+1}$ is orthogonal group $O(n+1)$ and translation \cite{K}, namely affine group. Hence
the $\vec{\tau}=A\vec{r}+\vec{b}$ generated by its Lie algebra is always the solution to homogeneous linearized equation,
where $A\in o(n+1)$ is a skew matrix and $\vec{b}$ is a constant vector. Such $\vec{\tau}$ is called a trivial solution to \eqref{1.2}.
For $n=2,$ it's equivalent to $\vec{\tau}=\vec{a}\times r+\vec{b}$ for any constant $\vec{a}$ and $\vec{b}$.
\begin{defi}
The surface is infinitesimally rigid if \eqref{1.2} has only trivial solutions.
\end{defi}
In the present paper we will revisit several kinds of rigid surfaces and give new proof which is based on
the equivalence of isometric embedding equation \eqref{1.1}, Gauss-Codazzi equations and Darboux equation.

For the case of $n=2,$ Cohn-Voseen and Blaschke \cite{CV,B} proved
\begin{theo}\label{CV}
Let $\mathcal{M}$ be a smooth closed surface with nonnegative curvature and let the vanishing set of the curvature has no interior points. Then $\mathcal{M}$
is globally rigid.
\end{theo}
\begin{theo}\label{CB}
Let $\mathcal{M}$ be a smooth closed surface with nonnegative curvature and let the vanishing set of the curvature has no interior points. Then $\mathcal{M}$
is infinitesimally rigid.
\end{theo}

Another rigid surface is Alexandrov's annuli \cite{A}.
\begin{defi}
The $2$ dimensional multiply-connected Riemannian manifold  $(\mathcal{M},g)$ satisfies Alexandrov's assumption:
\begin{eqnarray}
&&K>0, \text{ in }\mathcal{M} ,\\
&& \int_{\mathcal{M}}K dg=4\pi,\nonumber\\
&& K=0, \nabla K \neq 0 \text{ on }\partial \mathcal{M}.\nonumber
\end{eqnarray}

If $\vec{r}$ is the isometric embedding of $(\mathcal{M},g)$ in $\mathbb{R}^3,$
we call $\vec{r}$  Alexandrov's annuli.
\end{defi}
The following rigidity theorems are due to Alexandrov \cite{A} and Yau \cite{Y} respectively,
\begin{theo}\label{A}
 Alexandrov's annuli $\vec{r}$ is globally rigid.
\end{theo}
\begin{theo}\label{Yau}
 Alexandrov's annuli $\vec{r}$ is infinitesimally rigid.
\end{theo}

For the case of $n\geq 3,$ Dajczer-Rodriguez \cite{DR} proved

\begin{theo}\label{DR}
If the rank of the matrix $(h_{ij})$ is greater than $2$, where $h=h_{ij}dx^idx^j$ is the second fundamental form, then the hypersurface is globally and infinitesimally rigid.
\end{theo}

\begin{rema}
Compared with the case of $n=2,$ Dajczer-Rodriguez's theorem is local without any topological restriction on $\mathcal{M}.$
\end{rema}

\section{Set up and formulation}

Before discussing the rigidity of Alexandrov's annuli, we need some geometric preliminaries.

We use the geodesic coordinates $(s,t)=(x^1,x^2)$ based on $\partial \mathcal{M}$,
\begin{eqnarray}
&& g=dt^2+B^2ds^2,\nonumber\\
&&B(s,0)=1,B_t(s,0)=k_g,\nonumber
\end{eqnarray}
where $B(s,t)$ is a sufficiently smooth function and $B(s,t)$ is periodic in $s$,
and $k_g$ is geodesic curvature.

Under the geodesic coordinates, Alexandrov proved \cite{A,HH}
\begin{lemm}\label{HH}
For Alexandrov's annuli, the coefficients of the second
fundamental form of $ {\vec{r}}$, $ L,M$ and $N$ satisfy: at $t=0,$
\begin{eqnarray}
&&L=M=0,\nonumber\\
&&\partial_tL=\sqrt {K_tB_t},N=\sqrt {\frac {K_t}{B_t}}.
\end{eqnarray}
\end{lemm}
Since on $\partial \mathcal{M}$, $d\vec{n}=0$ and $k_n=0$ where $\vec{n}$ and $k_n$ are normal vector and normal curvature respectively,
 we have
\begin{lemm}\label{planar boundary}
The components of boundary $\vec{r}(\partial \mathcal{M})$ are some planar curves $\sigma_k, 1\leq k\leq m, $ which are determined
completely by its metric, and lies on the plane $\pi_k$ tangential to $\vec{r}$ along $\sigma_k.$
\end{lemm}

At the same time, Dong \cite{D} proved the following

\begin{lemm}\label{Dong}If there exist sufficiently smooth isometric embedding
$$\vec{r}:\mathcal{M}\rightarrow R^3, g=d\vec{r}^2,$$
then we have
\begin{eqnarray}
&&K_tk_{g}>0,\text{ on }\partial \mathcal{M},\\
&&\oint_{\sigma_k} k_gds=2\pi, \\
&&\oint_{\sigma_k}\exp{(\sqrt{-1}\int_{0}^{s}k_gd\theta)}ds=0.
\end{eqnarray}
\end{lemm}

In what follows we will formulate for the rigidity.

Let
\begin{eqnarray}\label{2.5}
\rho=\frac{1}{2}\vec{r}\cdot\vec{r},  \text{ and }\tilde{\rho}=\frac{1}{2}\tilde{r}\cdot\tilde{r},\nonumber\\
\mu=\vec{r}\cdot\vec{n},  \text{ and } \tilde{\mu}=\tilde{r}\cdot\tilde{n},
\end{eqnarray}
we have
\begin{eqnarray}
\vec{r}=g^{ij}\rho_i\vec{r}_{j}+\mu \vec{n},\nonumber\\
\mu^2=2\rho-|\nabla\rho|^2,\nonumber
\end{eqnarray}
and
\begin{eqnarray}
h_{ij}\mu=\rho_{i,j}-g_{ij};\tilde{h}_{ij}\tilde{\mu}=\tilde{\rho}_{i,j}-g_{ij},\\
\det(h_{ij})=\det(\tilde{h}_{ij})=K|g|,
\end{eqnarray}
where $h=h_{ij}dx^idx^j, \tilde{h}=\tilde{h}_{ij}dx^idx^j$ are the second fundamental forms respectively, $K$ is the Gaussian curvature.

Let $W_{ij}=\tilde{h}_{ij}-h_{ij}$ and $\Phi=\tilde{\rho}-\rho,$ by (2.6)-(2.7) we have
\begin{eqnarray}
(\tilde{h}_{ij}-W_{ij})\mu=\tilde{\rho}_{i,j}-\Phi_{i,j}-g_{ij}=\tilde{h}_{ij}\tilde{\mu}-\Phi_{i,j};\\
(h_{ij}+W_{ij})\tilde{\mu}=\rho_{i,j}+\Phi_{i,j}-g_{ij}=h_{ij}\mu+\Phi_{i,j};\\
\det(\tilde{h}_{ij}-W_{ij})=\det(h_{ij}+W_{ij}).
\end{eqnarray}

Taking the difference of (2.8)-(2.9) and the two sides of (2.10) yields
\begin{eqnarray}
W_{ij}(\mu+\tilde{\mu})=2\Phi_{i,j}+(h_{ij}+\tilde{h}_{ij})(\mu-\tilde{\mu});\\
(h_{11}+\tilde{h}_{11})w_{22}+(h_{22}+\tilde{h}_{22})w_{11}-2(h_{12}+\tilde{h}_{12})w_{12}=0.
\end{eqnarray}

Let $\bar{h}=h+\tilde{h}, \bar{h}_{ij}=h_{ij}+\tilde{h}_{ij},$ then
\begin{eqnarray}\label{2.13}
W_{ij}=\frac{2\Phi_{i,j}+\bar{h}_{ij}(\mu-\tilde{\mu})}{\mu+\tilde{\mu}}.
\end{eqnarray}
Gauss-Codazzi equations says
\begin{eqnarray}
&&\bar{h}^{ij}W_{ij}=0,\\
&&W_{ij,k}=W_{ik,j},
\end{eqnarray}
where  $(\bar{h}^{ij})=(\bar{h}_{ij})^{-1}.$

There exists an orthogonal mapping which sends the frame $\{r_1,r_2,n\}$ to $\{\tilde{r}_1,\tilde{r}_2,\tilde{n}\}$.
Let the associated matrix be $A$, if  $h$ and $\tilde{h}$ coincide which means $A$ is constant, i.e. $W=W_{ij}dx^idx^j=0,$
$\vec{r}$ and $\tilde{r}$ differ from an isometry and so it's globally rigid.

For the solution to \eqref{1.2} $\vec{\tau},$ let
\begin{eqnarray}
u_i=\vec{n}\cdot\vec{\tau}_i
\end{eqnarray}
and
\begin{eqnarray}
w=\frac{1}{2\sqrt{|g|}}(\vec{r}_2\cdot\vec{\tau}_1-\vec{r}_1\cdot\vec{\tau}_2).
\end{eqnarray}
Note that $u_idx^i=\vec{n}\cdot d\vec{\tau}$ is a globally well defined $1-$ form, and $w$ is a well defined function, then we have
\begin{eqnarray}
&&\vec{\tau}_1=w\sqrt{|g|}g^{2i}\vec{r}_i+u_1\vec{n},\\
&&\vec{\tau}_2=-w\sqrt{|g|}g^{1i}\vec{r}_i+u_2\vec{n}.
\end{eqnarray}
 Then for
 $$\vec{Y}=\frac{u_2\vec{r}_1-u_1\vec{r}_2}{\sqrt{|g|}}+w\vec{n},$$
 \begin{eqnarray}\label{2.20}
 d\vec{\tau}=\vec{Y}\times d\vec{r},
 \end{eqnarray}
 we call $\vec{Y}$ the rotation vector.
Differentiating  the above equation, we have
$$d^2\vec{\tau}=d\vec{Y}\times d\vec{r}=0,$$ which implies  $d\vec{Y}$ is parallel to the tangent plane.
Let $\vec{Y}_k=g^{ij}w_{ki}n\times \vec{r}_j,k=1,2,$  where $w_{ij}dx^{i}dx^{j}$ is a symmetric tensor.
$d^2\vec{Y}=0$ means
\begin{eqnarray}\label{2.21}
&&h^{ij}w_{ij}=0,\\
&&w_{ij,k}=w_{ik,j},
\end{eqnarray}
where $h=h_{ij}dx^idx^j$ is the second fundamental form and $(h^{ij})=(h_{ij})^{-1}.$
\begin{rema}
We note that $\vec{r}$ is infinitesimally rigid if and only if (2.21)-(2.22) have only trivial solution $w_{ij}=0$ provided that $\mathcal{M}$ is simply connected. In fact
$w_{ij}=0$ implies that $\vec{Y}$ is a constant.
\end{rema}

Let
\begin{eqnarray}\label{2.23}
\vec{b}=\vec{\tau}-\vec{Y}\times\vec{r},\varphi=\vec{b}\cdot\vec{r}=\vec{r}\cdot\vec{\tau},
\end{eqnarray}
we have
\begin{eqnarray}
d\vec{b}=-d\vec{Y}\times\vec{r}\\
\vec{b}=g^{ij}\varphi_i\vec{r}_j+\frac{\varphi-g^{ij}\varphi_i\rho_j}{\mu}\vec{n}.
\end{eqnarray}
Combing (2.24)-(2.25), we have
\begin{eqnarray}\label{2.26}
w_{ij}&=&\frac{\varphi_{i,j}}{\mu}+\frac{h_{ij}2(\varphi-\nabla\varphi\cdot\nabla\rho)}{\mu^2}\nonumber\\
      &=&\frac{\varphi_{i,j}}{\mu}+\frac{h_{ij}\nu}{\mu^2}.
\end{eqnarray}
If the support function $\mu\neq 0,$  $w_{ij}=0$ if and only if $\vec{b}$ is constant since $\vec{r}_1\times\vec{r},\vec{r}_2\times\vec{r}$
are linearly independent, i.e. $(\vec{r}_1\times\vec{r})\times(\vec{r}_2\times\vec{r})=\sqrt{|g|}\vec{r}\cdot\vec{n}=\sqrt{|g|}\mu.$
 For convex surface, by a translation we can assume the support function $\mu>0.$  Throughout the paper $\mu>0$ if not specified.

\section{The rigidity of surfaces in $\mathbb{R}^3$}

In this section will reprove Theorem \ref{CV}, Theorem \ref{A}  and Theorem \ref{CB}, Theorem \ref{Yau}. The main ideas are from an unpublished note \cite{LW}.

To prove Theorem \ref{CV} and Theorem \ref{A}, we introduce the following inner product,
for any two $(0,2)-$ symmetric tensors $\alpha=\alpha_{ik}dx^i\otimes dx^k, \beta=\beta_{jl}dx^j\otimes dx^l$,
\begin{equation}\label{inner product}
(\alpha,\beta)=\int_{\mathcal{M}}\frac{\det(\bar{h})}{\det(g)}\bar{h}^{ij}\bar{h}^{kl}\alpha_{ik}\beta_{jl}(\mu+\tilde{\mu})dV_g.
\end{equation}
Since $\bar{h}=h+\tilde{h}$ is positive definite,  we can view $\bar{h}=\bar{h}_{ij}dx^i\otimes dx^j$ as a Riemannian metric defined on $\mathcal{M}.$ Then
the cotangent bundle is endowed with the metric
\begin{equation}\label{ cotangent metric}
<dx^i,dx^j>=\bar{h}^{ij},
\end{equation}
and the metric induces a metric on the tensor bundle $T^{*}\mathcal{M}\otimes T^{*}\mathcal{M},$
\begin{equation}\label{tensor bundle metric}
<dx^i\otimes dx^k,dx^j\otimes dx^l>=\bar{h}^{ij}\bar{h}^{kl}.
\end{equation}
Note that $\det(\bar{h})(\mu+\tilde{\mu})>0 $ on $\mathcal{M},$ the integral defined by \eqref{inner product} is an inner product.

In what follows we will show the tensor $W=0$ by $(W,W)=0,$ where $W=W_{ij}dx^idx^j$ is the solution to (2.14)-(2.15), hence prove Theorem \ref{CV} and Theorem \ref{A}.
\begin{proof} A direct computation shows
\begin{eqnarray}\label{3.1}
&&(W,W)\nonumber\\
&=&\int_{\mathcal{M}}\frac{\det(\bar{h})}{\det(g)}\bar{h}^{ij}\bar{h}^{kl}W_{ik}W_{jl}(\mu+\tilde{\mu}) \nonumber\\
&=&\int_{\mathcal{M}}\frac{\det(\bar{h})}{\det(g)}\bar{h}^{ij}\bar{h}^{kl}(2\Phi_{i,k}+\bar{h}_{ik}(\mu-\tilde{\mu}))W_{jl} \nonumber\\
&=&\int_{\mathcal{M}}\frac{\det(\bar{h})}{\det(g)}\bar{h}^{ij}\bar{h}^{kl}2\Phi_{i,k}W_{jl}\nonumber\\
&=&2\int_{\partial\mathcal{M}}X\cdot\vec{\nu}dV_{\partial\mathcal{M}}-2\int_{\mathcal{M}}\Phi_{i}(\frac{\det(\bar{h})}{\det(g)}\bar{h}^{ij}\bar{h}^{kl}W_{jl})_{,k},
\end{eqnarray}
where $X=\det(\bar{h})\bar{h}^{ij}\bar{h}^{kl}\varphi_{i}W_{jl}\frac{\partial}{\partial x^k}$ and $\vec{\nu}$ is
outward normal along the $\partial{\mathcal{M}}.$ In the third equality we use $\bar{h}^{ij}W_{ij}=0,$ and the fourth equality is an application of divergence theorem.

For $i=1,$
\begin{eqnarray}\label{3.2}
&&(\det(\bar{h})\bar{h}^{ij}\bar{h}^{kl}w_{jl})_{,k}\nonumber\\
&=&(\bar{A}_{11}\bar{h}^{1l}W_{1l}+\bar{A}_{12}\bar{h}^{1l}W_{2l})_{,1}+(\bar{A}_{11}\bar{h}^{2l}w_{1l}+\bar{A}_{12}\bar{h}^{2l}W_{2l})_{,2}\nonumber\\
&=&(-\bar{A}_{11}\bar{h}^{2l}W_{2l}+\bar{A}_{12}\bar{h}^{1l}W_{2l})_{,1}+(\bar{A}_{11}\bar{h}^{2l}W_{1l}-\bar{A}_{12}\bar{h}^{1l}W_{1l})_{,2}\nonumber\\
&=&(-\bar{h}_{22}\bar{h}^{2l}W_{2l}-\bar{h}_{12}\bar{h}^{1l}W_{2l})_{,1}+(\bar{h}_{22}\bar{h}^{2l}W_{1l}+\bar{h}_{12}\bar{h}^{1l}W_{1l})_{,2}\nonumber\\
&=&-(\delta_{2}^{l}W_{2l})_{,1}+(\delta_{2}^{l}W_{1l})_{,2}\nonumber\\
&=&W_{21,2}-W_{22,1}\nonumber\\
&=&0,
\end{eqnarray}
where $\bar{A}_{ij}=\det(\bar{h})\bar{h}^{ij}$ is the cofactor of $\bar{h}$.
In the second equality and the last equality, we have used $\bar{h}^{ij}W_{ij}=0, W_{ij,k}=W_{ik,j}.$
Similarly, for $i=2,$ we also have $$(\det(\bar{h})\bar{h}^{ij}\bar{h}^{kl}W_{jl})_{,k}=0.$$
If $\mathcal{M}=\mathbb{S}^2,$ in the integral by parts the boundary term vanishes; if $\mathcal{M}$ is Alexandrov's annuli, on the boundary $W=0$ by Lemma \ref{HH} hence the boundary term vanishes too. Both of the two terms in \eqref{3.1} vanish,
$(W,W)=0,$ $W\equiv 0.$
\end{proof}

To prove Theorem \ref{CB} and Theorem \ref{Yau}, we introduce the following inner product,
for any two $(0,2)-$ symmetric tensors $\alpha=\alpha_{ik}dx^idx^k, \beta=\beta_{jl}dx^jdx^l$,
$$(\alpha,\beta)=\int_{\mathbb{S}^2}\frac{\det(h)}{\det(g)}h^{ij}h^{kl}\alpha_{ik}\beta_{jl}\mu dV_g.$$
In what follows we will show the tensor $w=0$ by $(w,w)=0,$ where $w=w_{ij}dx^idx^j$ is the solution to (2.21)-(2.22), hence prove Theorem \ref{CB}  and Theorem \ref{Yau}.
\begin{proof} A direct computation shows
\begin{eqnarray}
&&(w,w)\nonumber\\
&=&\int_{\mathcal{M}}\frac{\det(h)}{\det(g)}h^{ij}h^{kl}w_{ik}w_{jl}\mu \nonumber\\
&=&\int_{\mathcal{M}}\frac{\det(h)}{\det(g)}h^{ij}h^{kl}(\varphi_{i,k}+\frac{h_{ik}\nu}{\mu})w_{jl} \nonumber\\
&=&\int_{\mathcal{M}}\frac{\det(h)}{\det(g)}h^{ij}h^{kl}\varphi_{i,k}w_{jl}\nonumber\\
&=&\int_{\partial\mathcal{M}}X\cdot\vec{\nu}dV_{\partial\mathcal{M}},-\int_{\mathcal{M}}\varphi_{i}(\frac{\det(h)}{\det(g)}h^{ij}h^{kl}w_{jl})_{,k}
\end{eqnarray}
where $X=\det(h)h^{ij}h^{kl}\varphi_{i}w_{jl}\frac{\partial}{\partial x^k}$ and $\vec{\nu}$ is
outward normal along the $\partial{\mathcal{M}}.$

If $\mathcal{M}=\mathbb{S}^2,$  a similar argument in \eqref{3.2} yields $(w,w)=0,$ $w\equiv 0.$

If  $\mathcal{M}$ is Alexandrov's annuli, we have
\begin{eqnarray}\label{3.4}
(w,w)=\int_{\partial\mathcal{M}}X\cdot\vec{\nu}dV_{\partial\mathcal{M}},
\end{eqnarray}

Note the right hand side of \eqref{3.4} is invariance under coordinate change. So we use geodesic coordinates based on
$\partial{\mathcal{M}}.$ Without loss of generality, we merely consider the case of $\mathcal{M}$ is a disk, and then
$\partial{\mathcal{M}}$ is a planar curve denoted by $\sigma.$
On the boundary, we have $h_{11}=h_{12}=0$, $w_{11}=0$ and $\mu$ is constant.
\begin{eqnarray}\label{3.5}
&&\int_{\partial\mathcal{M}}X\cdot\vec{\nu}dV_{\partial\mathcal{M}}\nonumber\\
&=&\int_{\sigma}\det(h)h^{ij}h^{2l}\varphi_{i}w_{jl}ds\nonumber\\
&=&\int_{\sigma}\det(h)\varphi_{1}(h^{11}h^{22}w_{12}+h^{12}h^{21}w_{21}+h^{11}h^{21}w_{11}+h^{12}h^{22}w_{22})\nonumber\\
&+&\int_{\sigma}\det(h)\varphi_{2}(h^{21}h^{22}w_{12}+h^{22}h^{21}w_{21}+h^{21}h^{21}w_{11}+h^{22}h^{22}w_{22})\nonumber\\
&=&\int_{\sigma}\det(h)\varphi_{1}(h^{11}h^{22}w_{12}+h^{12}(h^{21}w_{21}+h^{11}w_{11}+h^{22}w_{22}))\nonumber\\
&=&\int_{\sigma}\det(h)\varphi_{1}(h^{11}h^{22}w_{12}-h^{12}h^{12}w_{21})\nonumber\\
&=&\int_{\sigma}\varphi_{1}w_{21},
\end{eqnarray}
where in the third equality we use the fact $h_{11}=h_{12}=0$, $w_{11}=0$ and in the fourth equality we use $h^{ij}w_{ij}=0.$

In what follows we will show
\begin{eqnarray}
\frac{1}{\mu}\oint_{\sigma}\varphi_{s}Fds\leq 0,\nonumber
\end{eqnarray}
where $F=w_{12}\mu.$

Recall on the boundary  $\sigma,$ $h_{11}=h_{12}=0,$  $w_{11}=0,$ and $\Gamma_{11}^{2}=-\Gamma_{12}^{1}=k_g,\Gamma_{11}^{1}=\Gamma_{12}^{2}=0.$ By \eqref{2.26},  we have on the boundary
\begin{eqnarray}
\left\{\begin{matrix}\label{3.6}
\varphi_{ss}&=&k_g\varphi_{t}\\
\varphi_{ts}&=&-k_g\varphi_{s}+F
\end{matrix}\right.,
\end{eqnarray}
which is nothing else but an ODE of $\varphi_s$ and $\varphi_t$.  We can rewrite \eqref{3.6} in complex form
\begin{eqnarray}
\frac{d}{ds}(\varphi_s+\sqrt{-1}\varphi_t)+\sqrt{-1}k_g(\varphi_s+\sqrt{-1}\varphi_t)=\sqrt{-1}F.\nonumber
\end{eqnarray}
For convenience, we introduce new variable $\theta=\int_{0}^{s}k_g\in [0,2\pi]$ and
let $c_1=\varphi_{s}(0), c_2=\varphi_{t}(0).$ Then the solution to \eqref{3.6} is
\begin{eqnarray}\label{3.7}
\varphi_s(\theta)=-\cos\theta(u(\theta)-c_1)+\sin\theta (v(\theta)+c_2),
\end{eqnarray}
where $f=\frac{F}{k_g}$ and
\begin{equation*}
u(\theta)=\int_{0}^{\theta}f(x)\sin x dx,v(\theta)=\int_{0}^{\theta}f(x)\cos x dx.
\end{equation*}

Suppose the boundary lies on the plane $z=0,$ by the motion of moving frame we have on the boundary
\begin{eqnarray}
\left\{\begin{matrix}\label{curve equation}
\vec{r}_{ss}&=&k_gr_{t}\\
\vec{r}_{ts}&=&-k_g\vec{r}_{s}
\end{matrix}\right..
\end{eqnarray}
It's easy to check
\begin{equation}\label{the tangent}
\begin{array}{ccc}
  \vec{r}_{s}(\theta) & = & \left( \begin{matrix}
   \cos(\theta+\alpha), & \sin(\theta+\alpha),  &0
\end{matrix} \right) \\
   \vec{r}_{t}(\theta) & = & \left( \begin{matrix}
  -\sin(\theta+\alpha), & \cos(\theta+\alpha),  &0
\end{matrix} \right)
\end{array},
\end{equation}
where $\alpha$ is a fixed constant.

In fact (2.4) follows from $\vec{r}_{s}(2\pi)=\vec{r}_{s}(0).$
Note that on $\sigma,$ $\vec{Y}_s=-w_{12}\vec{r}_s$ and $\mu$ is constant. By $\int_{\sigma}\vec{Y}_s=0,$
we have $u(2\pi)=v(2\pi)=0.$

Since $\oint_{\sigma}\varphi_sds=0,$
\begin{eqnarray}\label{3.8}
\oint_{\sigma}\varphi_sds&=&\int_{0}^{2\pi}\varphi_s(\theta)\frac{1}{k_g}d\theta\nonumber\\
&=&\int_{0}^{2\pi}(-\cos\theta u(\theta) +\sin\theta v(\theta))\frac{1}{k_g}d\theta\nonumber\\
&=&0,
\end{eqnarray}
where we use (2.3)-(2.4).

Hence
\begin{eqnarray}\label{3.9}
&&\oint_{\sigma}\varphi_{s}Fds\nonumber\\
&=&\int_{0}^{2\pi}-f\cos\theta(u(\theta)-c_1)+f\sin\theta (v(\theta)+c_2)d\theta\nonumber\\
&=&\int_{0}^{2\pi}-f\cos\theta u(\theta)+f\sin\theta v(\theta)d\theta\nonumber\\
&=&\int_{0}^{2\pi}-v'(\theta)u(\theta)+v(\theta)u'(\theta)d\theta\nonumber\\
&=&2\int_{0}^{2\pi}-v'(\theta)u(\theta).
\end{eqnarray}
We define a new closed planar curve $\Gamma$ by parameter equations
\begin{equation}\label{reference curve}
\begin{array}{ccc}
x_1(\theta)&=&\int_{0}^{\theta}\frac{\cos x}{k_g(x)}dx\\
x_2(\theta)&=&\int_{0}^{\theta}\frac{\sin x}{k_g(x)}dx\
\end{array}.
\end{equation}
A direct computation shows the curvature of $\Gamma$ is $k_g$ and the area bounded by the curve is
\begin{eqnarray}
S&=&-\oint_{\Gamma}x_2dx_1\nonumber\\
&=&\int_{0}^{2\pi}\frac{\cos\theta}{k_g(\theta)}\int_{0}^{\theta}\frac{\sin x}{k_g(x)}dx d\theta\nonumber\\
&>&0.\nonumber
\end{eqnarray}
And we introduce two new functions
\begin{eqnarray}
U(\theta)&=&u(\theta)+C\int_{0}^{\theta}\frac{\sin x}{k_g(x)}dx,\nonumber\\
V(\theta)&=&v(\theta)+C\int_{0}^{\theta}\frac{\cos x}{k_g(x)}dx,\nonumber\
\end{eqnarray}
where $$C=-\frac{u(\pi)}{\int_{0}^{\pi}\frac{\sin x}{k_g(x)}dx},$$ then we have $U'(\theta)\cot\theta=V'(\theta)$ and $U(0)=U(\pi)=0.$

Therefore
\begin{eqnarray}\label{3.10}
&&2\int_{0}^{2\pi}-V'(\theta)U(\theta)d\theta\nonumber\\
&=&2\int_{0}^{2\pi}-U(\theta)U'(\theta)\cot\theta d\theta\nonumber\\
&=&-\int_{0}^{2\pi}\sec^2\theta U^2(\theta)d\theta\nonumber\\
&\leq &0
\end{eqnarray}
and integral by parts yields
\begin{eqnarray}\label{3.11}
&&\int_{0}^{2\pi}-V'(\theta)U(\theta)d\theta\\
&=&\int_{0}^{2\pi}-(v'(\theta)+C\frac{\cos\theta}{k_g(\theta)})(u(\theta)+C\int_{0}^{\theta}\frac{\sin x}{k_g(x)}dx)d\theta\nonumber\\
&=&\int_{0}^{2\pi}-v'(\theta)u(\theta)d\theta+C\int_{0}^{2\pi}(v(\theta)\sin\theta-\cos\theta u(\theta))\frac{1}{k_g}d\theta
-C^2\oint_{\Gamma}x_2dx_1\nonumber\\
&=&\int_{0}^{2\pi}-v'(\theta)u(\theta)d\theta+C^2S,\nonumber
\end{eqnarray}
where in the third equality we use \eqref{3.8}.

Combing \eqref{3.9}-\eqref{3.11}, we have
\begin{eqnarray}\label{3.12}
\oint_{\sigma}\varphi_{s}Fds\leq 0,
\end{eqnarray}
and then
\begin{eqnarray}\label{3.13}
0\leq (w,w)\leq \int_{\sigma}\varphi_{1}w_{21}=\frac{1}{\mu}\oint_{\sigma}\varphi_{s}Fds\leq 0.
\end{eqnarray}
\end{proof}

In what follows we give another proof of Theorem \ref{CB} and Theorem \ref{Yau}. The proof is more geometric than above, correspondingly for Theorem \ref{Yau} we restrict that the component number of boundary of Alexandrov's positive annuli is $1$ (disk) or $2$ (annulus). We need the following lemma
\begin{lemm}\label{le10}
For any vector valued $\vec{E}:\mathcal{M}\mapsto \mathbb{R}^3$ satisfying
\begin{eqnarray}\label{3.14}
d\vec{r}\cdot d\vec{E}=0,
\end{eqnarray}
the $1-$ form defined on $\mathcal{M}$ $$\omega=d\vec{Y}\cdot\vec{E}$$ is closed.
\end{lemm}
\begin{proof}
It is obvious that $\omega$ is a one form. Exterior differentiation yields
\begin{eqnarray}
d\omega&=&\p_j(\vec{Y}_k\cdot \vec{E})dx^j\wedge dx^k\nonumber\\
&=&(\vec{Y}_{kj}\cdot \vec{E}+\vec{Y}_k\cdot\vec{E}_j)dx^j\wedge dx^k\nonumber\\
&=&((\vec{Y}_{21}-\vec{Y}_{12})\cdot \vec{E}+(\vec{Y}_2\cdot \vec{E}_1-\vec{Y}_1\cdot \vec{E}_2))dx^1\wedge dx^2\nonumber\\
&=&(\vec{Y}_2\cdot \vec{E}_1-\vec{Y}_1\cdot \vec{E}_2)dx^1\wedge dx^2.\nonumber
\end{eqnarray}

By \eqref{3.14}, we have
\begin{eqnarray}
\left\{\begin{matrix}\label{3.15}
\vec{r}_1\cdot E_1&=&0\\
\vec{r}_2\cdot \vec{E}_2&=&0\\
\vec{r}_1\cdot \vec{E}_2+\vec{r}_2\cdot \vec{E}_1&=&0
\end{matrix}\right..
\end{eqnarray}
We can rewrite $d\vec{Y}$ as
\begin{eqnarray}
\left\{\begin{matrix}\label{3.16}
\vec{Y}_1&=&\frac{1}{\sqrt{\det g}}(-w_{12}\vec{r}_1+w_{11}\vec{r}_2)\\
\vec{Y}_2&=&\frac{1}{\sqrt{\det g}}(-w_{22}\vec{r}_1+w_{21}\vec{r}_2)
\end{matrix}\right..
\end{eqnarray}
Hence by \eqref{3.16} we get,
$$\vec{Y}_2\cdot \vec{E}_1-\vec{Y}_1\cdot \vec{E}_2=\frac{w_{21}}{\sqrt{\det g}}\vec{r}_2\cdot \vec{E}_1+\frac{w_{12}}{\sqrt{\det g}}\vec{r}_1\cdot \vec{E}_2=0,$$ $\omega$ is a closed one form.
\end{proof}

Case $1:$ $\mathcal{M}$ be a disk $D$ called Alexandrov's positive disk, $\vec{k}$ be the normal along the boundary $\sigma,$ and $\vec{i},\vec{j}$ and $\vec{k}$ form an orthogonal basis.
Assume $\vec{r}\cdot\vec{k}=0$ on the boundary $\sigma,$ and $\vec{r}\cdot\vec{k}>0$ at the interior points, we have
$$\vec{E}=\vec{k} \text{  or  } \vec{E}=\vec{i}\times\vec{r} \text{ or }\vec{E}=\vec{j}\times\vec{r}$$
satisfy \eqref{3.11}. Since $\vec{r}=g^{ij}\rho_i\vec{r}_j+\mu\vec{n}=g^{ij}\rho_i\vec{r}_j \perp \vec{n}$ on $\sigma,$ and $\vec{i}\perp \vec{n},\vec{j}\perp \vec{n}$, we have $\vec{E}$ is parallel $\vec{n}=\vec{k}$ and then $\omega=d\vec{Y}\cdot\vec{E}=0$ on $\sigma.$

For convenience, we write
$$\vec{Y}_k=a_k^l\vec{r}_l,$$ where $a_k^l$ is a $(1,1)$ tensor. The relationship between $w_{ij}$ and $a_k^l$ is released in \eqref{3.16}.
Since the first de Rham cohomology of disk is trivial, i.e. $H^1_{DR}(D)=0$, there exists some smooth function $\psi$ defined on the disk, such that
$$\omega=d\psi=\psi_kdx^k.$$ Hence we have,
\begin{eqnarray}\label{3.17}
\psi_k=\vec{Y}_k\cdot \vec{E}=a^l_k\vec{r}_l\cdot\vec{E}.
\end{eqnarray}
We will show $\psi$ is constant hence $\omega=0,$ which is one key step to prove Theorem \ref{Yau}.

It's worth pointing out that  the following idea is borrowed from \cite{GWZ} which proves the rigidity in prescribed curvature problem.
\begin{proof}
A simple computation shows
$$\psi_{k,j}=a^l_{k,j}\vec{r}_l\cdot \vec{E}+a_k^lh_{jl}\vec{n}\cdot \vec{E}+a_k^l\vec{r}_l\cdot \vec{E}_j.$$
Then
\begin{eqnarray}\label{3.18}
h^{kj}\psi_{k,j}&=&h^{kj}a^l_{k,j}\vec{r}_l\cdot \vec{E}+a^k_k\vec{n}\cdot \vec{E}+h^{kj}a_k^l\vec{r}_l\cdot \vec{E}_j\\
&=&h^{kj}a^l_{k,j}\vec{r}_l\cdot \vec{E}+(\frac{-w_{12}}{\sqrt{\det g}}+\frac{w_{21}}{\sqrt{\det g}})\vec{n}\cdot \vec{E}\nonumber\\
&&+h^{1k}a_k^2\vec{r}_2\cdot \vec{E}_1+h^{2k}a_k^1\vec{r}_1\cdot \vec{E}_2\nonumber\\
&=&h^{kj}a^l_{k,j}\vec{r}_l\cdot \vec{E}+(h^{1k}a_k^2-h^{2k}a_k^1)\vec{r}_2\cdot \vec{E}_1\nonumber\\
&=&h^{kj}a^l_{k,j}\vec{r}_l\cdot \vec{E}+\frac{h^{ij}w_{ij}}{\sqrt{\det g}}\vec{r}_2\cdot \vec{E}_1\nonumber\\
&=&h^{kj}a^l_{k,j}\vec{r}_l\cdot \vec{E}\nonumber.
\end{eqnarray}
By the Lemma 4 in \cite{GWZ}, we have,
$$(a^1_1)^2+(a^1_2)^2+(a^2_1)^2+(a^2_2)^2\leq -C\det(a_{i}^{j}).$$
We conclude that,
$$h^{kj}\psi_{k,j}=h^{kj}a^l_{k,j}\frac{B^m_l \psi_m}{\det a}.$$ Here $B^m_l$ is the cofactor of $a^m_l$.
We also have for $l=1$,
\begin{eqnarray}
h^{ij}a^1_{i,j}
&=&h^{11}a^1_{1,1}+h^{12}a^1_{1,2}+h^{21}a^1_{2,1}+h^{22}a^1_{2,2}\nonumber\\
&=&\frac{1}{\sqrt{\det g}}(-h^{11}w_{12,1}-h^{12}w_{12,2}-h^{21}w_{22,1}-h^{22}w_{22,2})\nonumber\\
&=&-\frac{1}{\sqrt{\det g}}(h^{11}w_{11,2}+h^{12}w_{12,2}+h^{21}w_{21,2}+h^{22}w_{22,2})\nonumber\\
&=&-\frac{1}{\sqrt{\det g}}h^{ij}w_{ij,2}=\frac{1}{\sqrt{\det g}}h^{ij}_{,2}w_{ij}\nonumber.
\end{eqnarray}
Similarly, we have,
$$h^{ij}a^2_{i,j}=\frac{1}{\sqrt{\det g}}h^{ij}_{,1}w_{ij}.$$
Hopf's strong maximum principle (seen in $\S 3.2$ Theorem 3.5 of \cite{GT}) tells us
$\psi$ is a constant function on the disk since on the boundary $\psi$ is a constant, hence
$$\omega=d\psi=0.$$

Let $$S=\{x|x\in \bar{D},\vec{n}\cdot\vec{k}=\pm 1, \text{ or } \vec{r}\cdot\vec{k}=0\},$$
we have in $D\setminus S,$ at least one of the following mixed products is nonzero
\begin{eqnarray}
\left\{\begin{matrix}
(\vec{i}\times\vec{r},\vec{k},\vec{n})&=&-(\vec{r}\cdot\vec{k})(\vec{n}\cdot\vec{i})\\
(\vec{j}\times\vec{r},\vec{k},\vec{n})&=&-(\vec{r}\cdot\vec{k})(\vec{n}\cdot\vec{j})
\end{matrix}\right..
\end{eqnarray}
Recall that
$$\vec{E}=\vec{k} \text{  or  } \vec{E}=\vec{i}\times\vec{r} \text{ or }\vec{E}=\vec{j}\times\vec{r},$$
since $\omega=d\vec{Y}\cdot\vec{E}$ and $d\vec{Y}\cdot\vec{n}=0,$ $d\vec{Y}=0$ in $D\setminus S.$ Note that $S$ is zero measured, by the continuity  $d\vec{Y}=0$ in $D.$
\end{proof}

Case $2:$  $\mathcal{M}$ is Alexandrov's positive annulus.
Lemma \ref{planar boundary} says the boundary consists of two planar curves . We will discuss
two different case respectively: Subcase $2.1$: the two boundary planes are parallel; Subcase $2.2$: the two
boundary planes are not parallel.

Different from case $1,$ we need some extra topology preliminary
\begin{lemm}\label{homology}
If $\vec{E}=\vec{a}\times r+\vec{b}$ for any constant $\vec{a}$ and $\vec{b}$ which is the trivial solution to \eqref{1.2}, we have
for any component of boundary $\sigma_k,1\leq k\leq m$
\begin{eqnarray}\label{3.11}
\oint_{\sigma_k}\omega=\oint_{\sigma_k}d\vec{Y}\cdot\vec{E}=0,
\end{eqnarray}
hence there exists  some smooth function $\psi$ defined on the $\mathcal{M},$ such that
$$\omega=d\psi.$$
\end{lemm}
\begin{proof}
Integral by parts yields
\begin{eqnarray}
 \oint_{\sigma_k}d\vec{Y}\cdot\vec{E}
 &=&\oint_{\sigma_k}d\vec{Y}\cdot (\vec{a}\times r+\vec{b})\nonumber\\
 &=&\vec{a}\cdot \oint_{\sigma_k}\vec{Y}\times d\vec{r}+\vec{b}\cdot\oint_{\sigma_k}d\vec{Y}\nonumber\\
 &=&\vec{a}\cdot \oint_{\sigma_k}d\vec{\tau}+\vec{b}\cdot\oint_{\sigma_k}d\vec{Y}\nonumber\\
 &=&0,
 \end{eqnarray}
 where we use \eqref{2.20}.
 \end{proof}

 For Subcase $2.1,$ let $\vec{k}$ be a unit vector in $\mathbb{R}^3$ which is parallel to the normals of the two boundary
planes and choose $\vec{E}=\vec{k},$ then $\omega=d\vec{Y}\cdot \vec{k}=0$ on $\p \mathcal{M}.$ In particular the normal derivative $\frac{\p \psi}{\p \vec{\nu}}=0.$
Similar with Case $1,$ by maximum principle on Neumann problem (seen in $\S 3.2$ Theorem 3.6 of \cite{GT}) $\psi$ is constant. Hence $d\psi=\vec{Y}_i\cdot\vec{k}dx^i=0$
\begin{eqnarray}\label{3.13}
\left(\begin{matrix}
a^{1}_{1}&a^{2}_{1}\\
a^{1}_{2}&a^{2}_{2}
\end{matrix}\right)
\left(\begin{matrix}
\vec{r}_{1}\cdot \vec{k}\\
\vec{r}_{2}\cdot \vec{k}
\end{matrix}\right)=\left(\begin{matrix}
0\\
0
\end{matrix}\right).
\end{eqnarray}
Note that on $\mathcal{M}$ at least one of $\vec{r}_{1}\cdot \vec{k},\vec{r}_{2}\cdot \vec{k}$ is not zero otherwise  $\vec{k}$ is parallel to some normal on $\mathcal{M},$ but as a convex surface, its Gauss map is one-to-one and  any normal on $\mathcal{M}$  differs from the normals on $\p \mathcal{M}$ therefore isn't parallel to $\vec{k}.$
Hence the coefficient determinant $\det(a^{i}_{j})=\frac{\det(w_{ij})}{\det(g)}=0,$ i.e. $\det(w_{ij})=\det(w)=0.$  \eqref{2.21} says $\mathrm{tr}_{h}(w_{ij})=\mathrm{tr}(h^{-1}w)=0,$ in addition
$\det(h^{-1}w)=\frac{\det(w)}{\det(h)}=0$,  then $h^{-1}w=0$ and $w=0$ because $h$ and $w$ are symmetric, i.e. $d\vec{Y}=0.$

 For Subcase $2.2,$ let the constant normals on $\sigma_1,\sigma_2$ be $\vec{n}(\sigma_1),\vec{n}(\sigma_2),$ and the constant support functions on
 $\sigma_1,\sigma_2$ be $\mu(\sigma_1),\mu(\sigma_2),$ we choose $\vec{E}$ as
 \begin{eqnarray}\label{3.14}
 \vec{E}=(\vec{n}(\sigma_1)\times \vec{n}(\sigma_2))\times (\vec{r}+c_1\vec{n}(\sigma_1)+c_2\vec{n}(\sigma_2))
 \end{eqnarray}
 where $c_1,c_2$ solves
 \begin{eqnarray}\label{3.15}
\left(\begin{matrix}
1& \vec{n}(\sigma_1)\cdot\vec{n}(\sigma_2)\\
\vec{n}(\sigma_1)\cdot\vec{n}(\sigma_2)&1
\end{matrix}\right)
\left(\begin{matrix}
c_1\\
c_2
\end{matrix}\right)=-\left(\begin{matrix}
\mu(\sigma_1)\\
\mu(\sigma_2)
\end{matrix}\right).
\end{eqnarray}
Since $\vec{n}(\sigma_1),\vec{n}(\sigma_2)$ are not parallel, the coefficient matrix in algebraic equation \eqref{3.15} is invertible and thereby
\eqref{3.15} is solvable.

Note that $\vec{r}=g^{ij}\rho_i\vec{r}_j+\mu\vec{n}$,  it's easy to check that on $\p\mathcal{M}=\cup_{k=1}^{2}\sigma_k,$ such $\vec{E}$ is parallel to normal.
 Then $\omega=d\vec{Y}\cdot \vec{E}=0$ on $\p \mathcal{M}.$

Similar to Subcase $2.1,$ if at least one of $\vec{r}_{1}\cdot \vec{E},\vec{r}_{2}\cdot \vec{E}$ is not zero, the tensor $w=w_{ij}dx^idx^j=0.$
We will see the set
\begin{eqnarray}\label{3.16}
S_p:=\{p\in \mathcal{M}, \vec{r}_{1}\cdot \vec{E}=0,\vec{r}_{2}\cdot \vec{E}=0\}
\end{eqnarray}
is of zero measure. If so, by the continuity $w=0$ at any point on $\mathcal{M}.$

Let $\vec{X}=\vec{r}+c_1\vec{n}(\sigma_1)+c_2\vec{n}(\sigma_2),$ define
\begin{eqnarray}\label{3.17}
\varphi_{\mathcal{M}}(p)=\vec{n}\cdot\vec{X},  p\in  \mathcal{M},
\end{eqnarray}
 we have $S_p$ is contained in the level set $\{p\in  \mathcal{M},\varphi_{\mathcal{M}}(p)=0\}$ since $\vec{n}$ is parallel to $\vec{E}=(\vec{n}(\sigma_1)\times \vec{n}(\sigma_2))\times\vec{X}$ on $S_p.$ We will check on
 the level set,  $\nabla\varphi_{\mathcal{M}}\neq 0$ if $\vec{X}\neq 0.$
 \begin{eqnarray}\label{3.18}
 \partial_i\varphi_{\mathcal{M}}&=&\vec{r}_i\cdot\vec{n}+\vec{X}\cdot \partial_i\vec{n}\nonumber\\
 &=&-\vec{X}\cdot h^{l}_{i}\vec{r}_l.
 \end{eqnarray}
Let $\vec{X}=a^{j}\vec{r}_j $ since on the level set $\vec{n}\cdot\vec{X}=0,$  if $\nabla\varphi_{\mathcal{M}}= 0,$ we have

 \begin{eqnarray}\label{3.19}
\left(\begin{matrix}
h^{1}_{1}&h^{2}_{1}\\
h^{1}_{2}&h^{2}_{2}
\end{matrix}\right)
\left(\begin{matrix}
a^1\\
a^2
\end{matrix}\right)=\left(\begin{matrix}
0\\
0
\end{matrix}\right),
\end{eqnarray}
hence $a^j=0$ and $\vec{X}=0.$  $\vec{r}$ is regular surface and  the translation $\vec{X}$ is regular too, then the $\{p\in \mathcal{M},\vec{X}(p)=0\}$ is finite. The level set $\{p\in  \mathcal{M},\varphi_{\mathcal{M}}(p)=0\}$ is zero measured, of course as its subset $S_p$ is too.

\begin{rema}
If $\mathcal{M}=\mathbb{S}^2,$ i.e. the case of closed convex surface, we choose
$$\vec{E}=\vec{k} \text{  or  } \vec{E}=\vec{i} \text{ or }\vec{E}=\vec{j},$$
similar but simpler argument yields $d\vec{Y}=0.$ Thus we complete the proof of Theorem \ref{CB}.
\end{rema}

As seen, the new proofs we give highlight the roles that the function $\rho$ defined in \eqref{2.5} and its linearized version $\varphi$ defined in \eqref{2.23} play. In fact we can extract all information from $\rho$ which satisfies Darboux equation in isometric embedding problem  as  we work on the support function in Minkowski problem.

\section{The rigidity of hypersurfaces in $\mathbb{R}^{n+1}, n\geq 3$ }

   Similarly in the case of higher dimension, for the equation \eqref{1.2}  we can assume that $$d\vec{\tau}=\vec{Y}\times d\vec{r},$$ for some vector $\vec{Y}\in G_{r}(n-1,n+1)\cong G_{r}(2,n+1)$, where $G_{r}(r,n+1)$ is Grassmannian.

   Let $$d\vec{r}=\vec{r}_jdx^j, d\vec{Y}=\vec{Y}_idx^i,1\leq i,j\leq n $$ and
    $$\vec{Y}_i=W_{i}^{\alpha\beta}e_{\alpha}\wedge e_{\beta},1\leq \alpha,\beta\leq n+1,$$ where $\vec{r}_{n+1}$ is the normal vector, and the basis $e_{\alpha}\wedge e_{\beta}$ in
   $G_{r}(2,n+1)$
    is defined by
\begin{eqnarray}\label{5.1}
e_{\alpha}\wedge e_{\beta}=\frac{1}{(n-1)!}\delta_{k_1k_2\cdots k_{n-1}\alpha\beta}^{12\cdots(n-1)n(n+1)}\vec{r}_{k_1}\wedge\vec{r}_{k_2}\wedge\cdots\wedge \vec{r}_{k_{n-1}},
\end{eqnarray}
where $\delta$ is generalized Kronecker symbol.
Obviously $e_{\alpha}\wedge e_{\beta}=-e_{\beta}\wedge e_{\alpha}$, we set
\begin{eqnarray}
W_{i}^{\alpha\beta}=-W_{i}^{\beta\alpha}.
\end{eqnarray}
 By $$d\vec{Y}\wedge d\vec{r}=0,$$
we have
\begin{eqnarray}
W_{i}^{\alpha\beta}e_{\alpha}\wedge e_{\beta}\wedge\vec{r}_j dx^{i}\wedge dx^{j}=0,
\end{eqnarray}
i.e.
\begin{eqnarray}
\frac{1}{(n-1)!}W_{i}^{\alpha\beta}\delta_{k_1k_2\cdots k_{n-1}\alpha\beta}^{12\cdots(n-1)n(n+1)}\vec{r}_{k_1}\wedge \vec{r}_{k_2}\wedge\cdots\wedge \vec{r}_{k_{n-1}}\wedge\vec{r}_j dx^{i}\wedge dx^{j}=0.
\end{eqnarray}
Define a basis $E_\gamma,1\leq \gamma\leq n+1$ in $G_{r}(n,n+1)\cong G_{r}(1,n+1)$ by
\begin{eqnarray}
\vec{r}_{k_1}\wedge \vec{r}_{k_2}\wedge\cdots\wedge \vec{r}_{k_{n-1}}\wedge \vec{r}_{j}=\delta_{k_1k_2\cdots k_{n-1}j\gamma}^{12\cdots(n-1) n(n+1)}E_\gamma,
\end{eqnarray}
hence
\begin{eqnarray}
&&\frac{1}{(n-1)!}W_{i}^{\alpha\beta}\delta_{k_1k_2\cdots k_{n-1}\alpha\beta}^{12\cdots(n-1)n(n+1)}\delta_{k_1k_2\cdots k_{n-1}j\gamma}^{12\cdots(n-1) n(n+1)}E_\gamma dx^{i}\wedge dx^{j}\\
&=&\frac{1}{(n-1)!}W_{i}^{\alpha\beta}\delta_{\alpha\beta}^{j\gamma}E_\gamma dx^{i}\wedge dx^{j}\nonumber\\
&=&0,\nonumber
\end{eqnarray}
i.e.
for fixed $i,j$ and $\gamma,$
\begin{eqnarray}
W_{i}^{\alpha\beta}\delta_{\alpha\beta}^{j\gamma}-W_{j}^{\alpha\beta}\delta_{\alpha\beta}^{i\gamma}=0,
\end{eqnarray}
hence
\begin{eqnarray}
W_{i}^{j\gamma}=W_{j}^{i\gamma}.
\end{eqnarray}

We claim
\begin{lemm}
if $1\leq i,j,\gamma\leq n$, then
\begin{eqnarray}
W_{i}^{j\gamma}=0.
\end{eqnarray}
\end{lemm}

For the left hand side of (4.9), by (4.2),(4.8)
\begin{eqnarray}
W_{i}^{j\gamma}&=&-W_{i}^{\gamma j}\\
&=&-W_{\gamma}^{ij}\nonumber\\
&=&W_{\gamma}^{ji}.\nonumber
\end{eqnarray}

And on the other hand, for the right hand side of (4.9), by (2,2) and (4.8)
\begin{eqnarray}
W_{j}^{i\gamma}&=&-W_{j}^{\gamma i}\\
&=&-W_{\gamma}^{ji},\nonumber
\end{eqnarray}
so
$$W_{i}^{j\gamma}=-W_{i}^{j\gamma}.$$

Hence we can rewrite $\vec{Y}_i$
\begin{eqnarray}
\vec{Y}_i=2W_{i}^{l(n+1)}e_{l}\wedge e_{n+1}.
\end{eqnarray}

At the same time note that for fixed $i,j$
\begin{eqnarray}
\vec{Y}_i\wedge\vec{r}_j&=&2W_{i}^{l(n+1)}\frac{1}{(n-1)!}\delta_{k_1k_2\cdots k_{n-1}l(n+1)}^{12\cdots(n-1)n(n+1)}\vec{r}_{k_1}\wedge\vec {r}_{k_2}\wedge\cdots\wedge\vec{r}_{k_{n-1}}\wedge \vec{r}_j\\
&=&2W_{i}^{j(n+1)}\frac{1}{(n-1)!}\delta_{k_1k_2\cdots k_{n-1}j(n+1)}^{12\cdots(n-1)n(n+1)}\vec{r}_{k_1}\wedge \vec{r}_{k_2}\wedge\cdots\wedge\vec{r}_{k_{n-1}}\wedge \vec{r}_j\nonumber\\
&=&2W_{i}^{j(n+1)}\sqrt{|g|}\vec{r}_{n+1},\nonumber
\end{eqnarray}
hence let $w_{ij}=2W_{i}^{j(n+1)}\sqrt{|g|}$, then the quadratic form $w_{ij}dx^{i}dx^{j}$ is globally well defined.

And we rewrite (4.12) as
\begin{eqnarray}
\vec{Y}_i=w_{il}\frac{1}{\sqrt{|g|}}e_{l}\wedge e_{n+1}.
\end{eqnarray}

In what follows we will compute the covariant derivative of $e_l\wedge e_{n+1}$.

At first we notice that
\begin{eqnarray}
e_l\wedge e_{n+1}=\sum_{k_1<k_2<\cdots<k_{n-1}}^{k_1,k_2,\cdots,k_{n-1}\neq l,n+1}\delta_{k_1k_2\cdots k_{n-1}l(n+1)}^{12\cdots(n-1)n(n+1)}\vec{r}_{k_1}\wedge \vec{r}_{k_2}\wedge\cdots\wedge \vec{r}_{k_{n-1}},
\end{eqnarray}
therefore
\begin{eqnarray}
&&(e_l\wedge e_{n+1})_{j}\\
&=&\sum_{k_1<k_2<\cdots<k_{n-1}}^{k_1,k_2,\cdots,k_{n-1}\neq l,n+1}\delta_{k_1k_2\cdots k_{n-1}l(n+1)}^{12\cdots(n-1)n(n+1)}(h_{k_1,j} \vec{r}_{n+1}\wedge \vec{r}_{k_2}\wedge\cdots\wedge \vec{r}_{k_{n-1}}\nonumber\\
&+&h_{k_2,j}\vec{r}_{k_1}\wedge \vec{r}_{n+1}\wedge\cdots\wedge \vec{r}_{k_{n-1}}
+\cdots+h_{k_{n-1},j}\vec{r}_{k_1}\wedge \vec{r}_{k_2}\wedge\cdots\wedge \vec{r}_{n+1}).\nonumber
\end{eqnarray}
Since
\begin{eqnarray}
&& \delta_{k_1k_2\cdots k_{n-1}l(n+1)}^{12\cdots(n-1)n(n+1)} \vec{r}_{n+1}\wedge \vec{r}_{k_2}\wedge\cdots\wedge \vec{r}_{k_{n-1}}\nonumber\\
&=&-\delta_{(n+1)k_2\cdots k_{n-1}lk_1}^{12\cdots(n-1)n(n+1)}\vec{r}_{n+1}\wedge \vec{r}_{k_2}\wedge\cdots\wedge \vec{r}_{k_{n-1}}\nonumber\\
&=&-\delta_{k_2\cdots k_{n-1}(n+1)lk_1}^{12\cdots(n-1)n(n+1)} \vec{r}_{k_2}\wedge\cdots\wedge \vec{r}_{k_{n-1}}\wedge \vec{r}_{n+1}\nonumber
\end{eqnarray}
and for $k_2<k_3<\cdots<k_{n-1}<n+1, k_2,k_3,k_{n-1},n+1\neq l,k_1,$ hence
\begin{eqnarray}
\delta_{k_1k_2\cdots k_{n-1}l(n+1)}^{12\cdots(n-1)n(n+1)} \vec{r}_{n+1}\wedge \vec{r}_{k_2}\wedge\cdots\wedge \vec{r}_{k_{n-1}}
=-e_{l}\wedge e_{k_1}.
\end{eqnarray}
Similarly we have
\begin{eqnarray}
(e_l\wedge e_{n+1})_{j}=\sum_{k\neq l,n+1}h_{kj}e_k\wedge e_l.
\end{eqnarray}
Thus
\begin{eqnarray}
\vec{Y}_{i,j}&=&\frac{1}{\sqrt{|g|}}(w_{il,j}e_l\wedge e_{n+1}+w_{il}\sum_{k\neq l,n+1}h_{kj}e_k\wedge e_l),\\
\vec{Y}_{j,i}&=&\frac{1}{\sqrt{|g|}}(w_{jl,i}e_l\wedge e_{n+1}+w_{jl}\sum_{k\neq l,n+1}h_{ki}e_k\wedge e_l).\nonumber
\end{eqnarray}
By compatibility $\vec{Y}_{i,j}=\vec{Y}_{j,i}$, we have
\begin{eqnarray}
&&w_{il,j}=w_{jl,i},\\
&&h_{kj}w_{il}-h_{lj}w_{ik}=h_{ki}w_{jl}-h_{li}w_{jk}.
\end{eqnarray}
\begin{rema}
(4.20) shows that $w_{ij}$ is Codazzi. In fact,(4.20)-(4.21) is  homogeneous linearized Gauss-Codazzi system.
\end{rema}
Similarly with the case of $n=2$,
\begin{eqnarray}
h^{ij}\vec{Y}_{i,j}=\frac{h^{ij}}{\sqrt{|g|}}w_{il,j}e_l\wedge e_{n+1},
\end{eqnarray}

hence for hypersurface in $\mathbb{R}^{n+1}$, we can use maximal principle to get the infinitesimal rigidity.
But we can make use of (4.21) to reprove Theorem \ref{DR}.

\begin{proof}

We want to show $w_{ij}=0$. In view that $w_{ij}dx^{i}dx^{j}$ is invariant under variable transformation, we consider the diagonal case,
i.e. $h_{ij}=0,i\neq j$, since at any point on the hypersurface we can diagonalize the matrix  $(h_{ij})$ by variable transformation.

If the rank of the matrix $(h_{ij})$ is greater than $2$, without loss of generality we can assume $h_{11},h_{22},h_{33}\neq 0.$
By (4.21),
\begin{eqnarray}
&&h_{11}w_{22}+h_{22}w_{11}=h_{12}w_{21}+h_{21}w_{12},\\
&&h_{11}w_{33}+h_{33}w_{11}=h_{13}w_{31}+h_{31}w_{13},\nonumber\\
&&h_{22}w_{33}+h_{33}w_{22}=h_{23}w_{32}+h_{32}w_{23}.\nonumber
\end{eqnarray}
Since $h_{ij}=0,i\neq j,$ (4.23) is just a linear system of $w_{11},w_{22},w_{33}$
\begin{eqnarray}
\left(
\begin{matrix}
h_{22}&h_{11}&0\\
h_{33}&0&h_{11}\\
0&h_{33}&h_{22}\\
\end{matrix}
  \right)
  \left(
\begin{matrix}
w_{11}\\
w_{22}\\
w_{33}\\
\end{matrix}
  \right)
=\left(
\begin{matrix}
0\\
0\\
0\\
\end{matrix}
  \right).
\end{eqnarray}
The coefficient matrix in (4.24) is invertible, hence $w_{11}=w_{22}=w_{33}$. For other $w_{ij}$, by (4.21)
\begin{eqnarray}
h_{11}w_{ij}+h_{ij}w_{11}=h_{1i}w_{j1}+h_{1j}w_{i1},
\end{eqnarray}
since $i\neq 1, j\neq 1$ and $w_{11}=0$, $h_{11}w_{ij}=0$.

As to the part of global rigidity, without loss of generality we assume the block  $H_3=(h_{ij})_{3\times 3}$ is full rank, then its adjoint matrix $H_{3}^{*}$ is full rank too. By Gauss equation, every element in $H_{3}^{*}$ is an entry of Riemannian curvature tensor which is totally determined by metric. Therefore  $H_{3}^{*}$ is intrinsic and we can recover $H_3$ from $H_{3}^{*}.$ $H^3$ is intrinsic too, and as we proceed in the part of infinitesimal rigidity the $H= (h_{ij})_{n\times n}$ is intrinsic too.
\end{proof}
In the proof of Theorem \ref{DR}, we just deal with the algebraic equations, Gauss equations or its linearized equations, so we can say Theorem \ref{DR} is algebraic.

\begin{ack}
The authors wish to thank  Professor Pengfei Guan and Professor Zhizhang Wang for their valuable suggestions and comments.  Part of the content also comes from Professor Wang's contribution. The first author wishes to thank China Scholarship Council for its financial support. The first author also would  like to thank McGill University for their hospitality.
\end{ack}

\end{document}